\documentclass[10pt,reqno]{amsart}
\usepackage{amssymb,amsmath,amsthm}
\usepackage{mathrsfs,pifont}
\usepackage{graphicx}
\usepackage[notcite,notref]{showkeys}
\usepackage{url,hyperref}
 \usepackage{color}
\usepackage{float}

\catcode`\@=11 \theoremstyle{plain}
\@addtoreset{figure}{section}
\renewcommand\thefigure{\thesection.\@arabic\c@figure}
\@addtoreset{table}{section}
\renewcommand\thetable{\thesection.\@arabic\c@table}

\newtheorem{thm}{Theorem}[section]

\newtheorem{lemma}{Lemma}[section]
\newtheorem{rem}{Remark}[section]

\def \af {\alpha}

\def \om {\omega}

\def\0{{\bf 0}}

\begin{document}

{\title[sp]
{A spectral collocation method for nonlocal diffusion equations}
\author[
	H. Tian,\; J. Zhang
	]{
		\;\; Hao Tian${}^1$,    \;\; Jing Zhang${}^{2}$
		}
	\thanks{${}^1$ School of Mathematics and Statistics, Ocean University of China, Qingdao, Shandong 266100, China \\
		\indent ${}^{2}$School of Mathematics and Statistics $\&$  Hubei Key Laboratory of Mathematical Sciences,  Central China Normal University, Wuhan 430079, China. 
}
\begin{abstract}
	Nonlocal diffusion model provides an appropriate description of the diffusion process of solute in the complex medium, which cannot be described properly by
	classical theory of PDE. However, the operators in the nonlocal diffusion models are
	nonlocal, so the resulting numerical methods generate dense or full stiffness matrices. This imposes significant computational and memory challenge for a nonlocal diffusion model.  
	In this paper, we develop a spectral collocation method for the nonlocal diffusion model and provide a rigorous error analysis which
	theoretically justifies the spectral rate of convergence provided that the kernel functions and the source functions are sufficiently smooth. Compared to finite difference methods and finite element methods, because of the high order convergence rates, the numerical cost of spectral collocation methods will be greatly decreased. Numerical results confirm the exponential rate of convergence.
\end{abstract}
 \maketitle

\baselineskip 12.8pt

\thispagestyle{empty}

\section{Introduction}
Nonlocal models given in terms of integral equations in spatial variables have received much attention in recent years
\cite{HuHaBo,IgnRos,MacSil,Sil00,Sil10,SilAsk,SilLeh,SilEptWec,SilZimAbe,SunLi,WecEmm1}, from both theoretical and computational point of view \cite{BobDua,CheGun,DuGun,DuGunLeh,DuJuTia,DuTian,DuTiaZha,DuZho,OteMadAgw,SelGunPar,ZhoDu}.  The modeling of central
nonlocal diffusion is based on a radially symmetric kernel function, in the nonlocal operator, which describes the statistical nature of a
stochastic process by assuming the probabilities of a particle moving in arbitrary directions are the same so that the processes is determined
by the dependence of jump rates on jump sizes.
Our goal is to study nonlocal-convection diffusion models (of integral-type) and their effective numerical solutions.

Let $I$ be a finite bar in ${\mathbb R}$. Without loss of generality, we take $I = [-1, 1]$. A
nonlocal operator $\mathcal L_\delta $ is defined as, for any function $u = u(x) : I\rightarrow \mathbb R$,
$${\mathcal L_\delta}u=\int_{B_{\delta}(x)}\big(u(y)-u(x)\big)\gamma(x,y)dy$$
with $B_\delta(x) = \{y \in {\mathbb R} : |y - x| < \delta\}$ denoting a neighborhood centered at $x$ of radius $\delta$
which is the horizon parameter, and $\gamma(x, y): x\times y \rightarrow \mathbb R$ being a symmetric nonlocal
kernel (influence function), i.e., $\gamma(x, y) = \gamma(y, x)$, and $\gamma(x, y) = 0$ if $y\notin  B_\delta(x)$. In this paper, we assume that the kernel function satisfies transition invariance, i.e.,$\gamma(x,y)=\gamma(|y-x|)$.

The following nonlocal diffusion model of one-dimension steady case is our
main subject of interests here:
\begin{equation}\label{eqxy}
	\left\{
	\begin{split}
		&{\mathcal L_\delta}u=f(x), \;x\in I:=[-1,1],\\
		&u(x)=g(x),\;x\in I_c:=(-1-\delta,-1)\cup(1,1+\delta).
	\end{split}\right.
\end{equation}
We refer to \cite{DuHua} for connections between
nonlocal diffusion equations and stochastic jump processes. The well-posedness of \eqref{eqxy} was studied in \cite{DDGL}.  Moreover, it is known that, under proper assumptions of the kernel function $\gamma(x,y)$, the nonlocal problem (\ref{eqxy}) converges to the local problem, as the horizon $\delta\rightarrow 0$.

Recently, there have been a lot of efforts on developing numerical methods for nonlocal diffusion models \eqref{eqxy}, like finite difference methods, finite element methods and meshless methods. However, spectral method for nonlocal  diffusion (ND) models  has received remarkably little attention. In fact spectral methods have been broadly applied to many integral equations like Voltera integral equations \cite{WEI201852,WEI201415} and fraction PDE\cite{YANG20171218,YANG2017} .  The purpose of this paper is to give new insights into spectral collocation algorithms for nonlocal diffusion model. The main contributions reside in the following aspects:
\begin{itemize}
\item We construct a spectral-collocation scheme for a nonlocal diffusion model. 

Firstly, because of their nonlocality, numerical methods for the
nonlocal diffusion models usually generate dense stiffness matrices in which the bandwidths increase to infinity as the
mesh size decreases to zero. Direct solvers are widely used in the nonlocal diffusion modeling, which have $O(N^2)$ memory
requirement to store the stiffness matrix and $O(N^3)$ computations to find the numerical solutions. If we apply spectral methods to nonlocal diffusion model, the stiff matrices would still be dense, however the scale of the stiff matrix would be greatly decreased, because of the high accuracy of spectral methods. In this way, the computation and storage 

Secondly, the accuracy and convergence of the numerical
methods for the nonlocal diffusion models depends heavily on the accurate evaluation of the integrals, which is defined on $B_{\delta}(x_i)\cap\Omega_j$, where $B_{\delta}(x_i)$ is the neighborhood of a collocation point $x_i$ and $\Omega_j$ is the supporting area of a basis function $\phi_j(x)$. For numerical methods like finite element methods, finite difference methods, and collocation methods, the basis function corresponds to a local supporting area. When $\Omega_j$ is on the edge of $B_{\delta}(x_i)$, $B_{\delta}(x_i)\cap\Omega_j$ could be highly irregular, which will degrade the accuracy of the numerical integration presented above. In contrast to finite element method and other related methods, the support area of spectral basis functions is the whole computational area. Thus, the intersection area is $B_{\delta}(x_i)$. The influence area is a regular sphere in $2$-D or a ball in $3$-D. So these numerical integrations in spectral method can achieve high order accuracy. Hence, this spectral-collocation scheme for nonlocal diffusion model can be extended to the high dimensional cases easily.

\item We provide a rigorous error analysis which
theoretically justifies the spectral rate of convergence. We also present more numerical
evidences to demonstrate this surprising convergence behavior.
\end{itemize}

The rest of paper is organized as follows. In Section \ref{pre}, we review basic properties of Legendre polynomials and the related quadrature rules, cardinal bases. In Section \ref{alg}, we introduce the spectral approaches for $1$-D nonlocal diffusion model.
Maximum principle for the convergence analysis will be provided in Section \ref{max}.
The convergence analysis in $L^\infty$ space will be given in Section \ref{con}. Numerical
experiments are carried out in Section \ref{exp} to verify the theoretical results obtained in Section \ref{con}.

\section{Mathematical preliminaries}\label{pre}
In this section, we introduce some notation and review the relevant properties of the Legendre  polynomials, the associate  quadrature
rules, cardinal basis (cf. \cite{Szego75}).
\subsection{Notation}
\begin{itemize}

\item Let $\om^{\af,\beta}(x)=(1-x)^{\af}(1+x)^{\beta}\; (\af,\beta>-1)$ be  the Jacobi weight function defined in $I:=(-1,1),$ and
let $L^2_{\om^{\af,\beta}}(I)$ be the Hilbert space with the inner product and norm
\begin{equation*}
(u,v)_{\af,\beta}=\int_{I} u(x) v(x)\om^{\af,\beta}(x) dx,\quad \|u\|_{\af,\beta}=\sqrt{(u,u)_{\af,\beta}}.
\end{equation*}
For any integer $r\ge 0,$ we define the weighted Sobolev space:
\begin{equation*}
H^r_{\om^{\af,\beta}}(I)=\big\{ u\in L^2_{\om^{\af,\beta}}(I) :  u^{(k)}\in L^2_{\om^{\af,\beta}}(I),\; 0\le k\le r \big\},
\end{equation*}
equipped with the norm  and semi-norm:
\begin{equation*}
\|u\|_{r,\om^{\af,\beta}}=\Big(\sum_{k=0}^r \big\|u^{(k)} \big\|_{\om^{\af,\beta}}^2 \Big)^{\frac 1 2},\quad
|u|_{r,\om^{\af,\beta}}=\big\|u^{(r)} \big\|_{\om^{\af,\beta}}.
\end{equation*}
For any real $r> 0,$
the space $H^r_{\om^{\af,\beta}}(I)$ and its norm $\|\cdot\|_{r,\om^{\af,\beta}}$ are defined by space interpolation as
 in \cite{Adam75}.
In particular,  we have $L^2_{\omega^{0,0}}(I)=H^0_{\omega^{0,0}}(I)$ and denote its  inner product and norm
by $(\cdot,\cdot)$ and $\|\cdot\|,$ respectively.
\item We use $\partial_x^k u(x)$ to denote the ordinary derivative $\frac{d^k }{dx^k} u(x) =u^{(k)}(x)$ for $ k\ge 1.$\\
\item We introduce the non-uniformly (or anisotropic) Jacobi-weighted Sobolev space:
\begin{equation*}
B_{\af,\beta}^m(I):=\{u:\partial_x^k u\in L^2_{\omega^{\af+k,\beta+k}(I),\;0\leq k\leq m}\},\quad m\in{\mathbb N},
\end{equation*}
equipped with the inner product, norm and semi-norm
\begin{equation*}
\begin{split}
&(u,v)_{B_{\af,\beta}^m}=\sum_{k=0}^m (\partial_x^k u, \partial_x^k v)_{\omega^{\af+k,\beta+k}},\\
&\|u\|_{B_{\af,\beta}^m}=(u,u)_{B_{\af,\beta}^m}^{1/2},\quad |u|_{B_{\af,\beta}^m}=\|\partial_x^m u\|_{\omega^{\af+m,\beta+m}}.
\end{split}
\end{equation*}
\item We denote by ${\mathbb P}_{N}$ the set of all algebraic polynomials of degree  $\leq N.$
\end{itemize}
\subsection{Legendre polynomials}
The Legendre polynomials, denoted by $L_{n}(x),$ are the
are mutually orthogonal with respect to $\omega^{0,0}=1,$ and   normalized so  that
\begin{equation*}
\int_{-1}^{1}L_{m}(x)L_{n}(x)
dx=\gamma_n\delta_{mn},\quad \gamma_n=\frac{2}{2n+1},
\end{equation*}
where $\delta_{mn}$ is Kronecker symbol. They satisfy  the three-term recurrence relation:
\begin{equation}\label{Leg}
\begin{split}
&(n+1)L_{n+1}(x)=(2N+1)x L_n(x)-nL_{n-1}(x),\quad n\geq {1},\\
& L_{0}(x)=1,\quad L_1(x)=x.
\end{split}
\end{equation}
\subsection{Legendre-Gauss-Type Quadratures  and cardinal basis }The Legendre-Gauss-type nodes and weights $\{x_j,\omega_j\}^N_{j=0}$ can be derived from the following formulas:
\begin{itemize}
\item For the Legendre-Gauss (LG) quadrature,
\begin{equation}
\begin{split}
&\{x_j\}_{j=0}^N \;{\rm{are\; the\; zeros \;of }}\; L_{N+1}(x);\\
&\omega_j=\frac{2}{(1-x_j^2)[L'_{N+1}(x_j)]^2},\quad 0\leq i \leq N,
\end{split}
\end{equation}
\item For the Legendre-Gauss-Radau (LGR) quadrature,
\begin{equation}
\begin{split}
&\{x_j\}_{j=0}^N \;{\rm{are\; the\; zeros \;of }}\; L_{N}(x)+L_{N+1}(x);\\
&\omega_j=\frac{1}{(N+1)^2}\frac{1-x_j}{[L_{N}(x_j)]^2},\quad 0\leq j \leq N,
\end{split}
\end{equation}
\item For the Legendre-Gauss-Lobatto (LGL) quadrature,
\begin{equation}
\begin{split}
&\{x_j\}_{j=0}^N \;{\rm{are\; the\; zeros \;of }}\; (1-x^2)L'_N(x);\\
&\omega_j=\frac{2}{N(N+1)}\frac{1}{[L_{N}(x_j)]^2},\quad 0\leq j \leq N,
\end{split}
\end{equation}
\end{itemize}
With the above quadrature nodes and weights, there holds
\begin{equation}
\int_{-1}^1 p(x)dx=\sum_{j=0}^N p(x_j)\omega_j,\quad \forall p\in {\mathbb P}_{2N+\delta},
\end{equation}
where $\delta=1,0,-1$ for LG, LGR and LGL, respectively. Moreover, the conventional choice of grid points for Legendre spectral-collocation methods, is the Legendre Gauss-Lobatto points.

The spectral-collocation method is usually implemented in the physical space by
seeking approximate solution in the form $u_N\in {\mathbb P}_{N}$ such that
\begin{equation*}
u_N(x)=\sum_{k=0}^Nu_N(x_k)h_k(x),
\end{equation*}
where $\{h_k\}$ are the Lagrange basis polynomials (also referred to as nodal basis
functions), i.e.,$h_k\in {\mathbb P}_{N}$ and $h_k(x_j) =\delta_{kj}$.  We  write
\begin{equation*}
h_k(x) =\sum_{p=0}^N\beta_{p,k}L_p(x),\quad 0\leq p, k \leq N.
\end{equation*}
and determine the coefficients $\beta_{p,k}$  from $h_k(x_j) =\delta_{kj},\; 0\leq k,j\leq N.$ More precisely,
\begin{equation*}
\beta_{p,k}=\frac{1}{\gamma_p}\sum_{i=0}^N h_k(x_i)L_p(x_i)\omega_i/\gamma_p=L_p(x_k)\omega_k/\gamma_p,=\end{equation*}
where
\begin{equation*}
\gamma_p=\sum_{i=0}^NL_p^2(x_i)\omega_i=(p+\frac{1}{2})^{-1},\quad {\rm for}\quad p<N
\end{equation*}
and $\gamma=(N + 1/2)^{-1}$ for the Gauss and Gauss-Radau formulas, and $\gamma_N = 2/N$ for the
Gauss-Lobatto formula.
\section{Numerical Algorithm}\label{alg}
Firstly, we make the change of variable
\begin{equation}\label{transf}
y=x+s,\; s\in [-\delta,\delta].
\end{equation}
under which \eqref{eqxy} is transformed into
\begin{equation}\label{eqs}
\int_{x-\delta}^{x+\delta}u(y)\gamma(|y-x|)dy-u(x)\int_{-\delta}^\delta\gamma(|s|)ds=f(x), \; x\in I:=[-1,1].
\end{equation}
Let $\{x_i\}_{i=0}^N$  be a set of Legendre-Gauss-Lobatto points, and a approximation to \eqref{eqxy} using a Legendre collocation
approach is
\begin{equation}\label{eqxs}
\left\{
\begin{split}
&{\rm{Find}}\;u_N\in {\mathbb P}_{N} \; {\rm such \; that}\\
&\int_{\Lambda} u_N(y)\gamma(|y-x_i|)dy-u_N(x_i)\int_{-\delta}^\delta\gamma(|s|)ds=f(x_i)-\int_{\Lambda_c} g(y)\gamma(|y-x_i|)dy,\;0\leq i\leq N,
\end{split}\right.
\end{equation}
where $\Lambda:=(x_i-\delta,x_i+\delta)\cap (-1,1)$ and $ \Lambda_c:=(x_i-\delta,x_i+\delta)/\Lambda.$ To compute the integral term in $\eqref{eqxs}$ accurately,  we will transfer the integral interval $\Lambda$ to a fixed interval $[-1, 1]$
and then make use of some appropriate quadrature rule. Firstly, define $\Lambda:=(a_i,b_i)=(x_i-\delta,x_i+\delta)\cap (-1,1)$ and  make a simple linear transformation:
\begin{equation*}
y=\frac{b_i-a_i}{2}t+\frac{a_i+b_i}{2},\quad t\in (-1,1).
\end{equation*}
Then \eqref{eqxs} becomes
\begin{equation}\label{eqxt}
\begin{split}
\frac{b_i-a_i}{2}\int_{-1}^1u_N(y(x_i,t))\gamma(x,y(x_i,t_j))dt-&u_N(x_i)\int_{-\delta}^\delta\gamma(|s|)ds\\
=&f(x_i)-\int_{\Lambda_c} g(y)\gamma(|y-x_i|)dy,\quad 0\leq i \leq N.
\end{split}
\end{equation}
We then approximate the integral term by a Legendre-Gauss type quadrature formula with the notes and weights denoted by $\{t_j,\omega_j\}^M_{j=0}$, leading to the Legendre collocation scheme (with numerical integration) for \eqref{eqxt}:
\begin{equation}\label{eqxtj}
\left\{
\begin{split}
&{\rm{Find}}\;u_N\in {\mathbb P}_{N} \; {\rm such \; that}&\\
&\frac{b_i-a_i}{2}\sum_{j=0}^Nu_N(y(x_i,t_j))\gamma(x,y(x_i,t_j))\omega_j-u_N(x_i)\int_{-\delta}^\delta\gamma(|s|)ds\\
&=f(x_i)-\int_{\Lambda_c} g(y)\gamma(|y-x_i|)dy,\quad 0\leq i\leq N.
\end{split}\right.
\end{equation}
We expand the approximate solution $u_N$ as
\begin{equation}
u_N(x)=\sum_{k=0}^Nu_kh_k(x).
\end{equation}
Inserting it into \eqref{eqxtj} leads to
\begin{equation}\label{coll}
\begin{split}
\frac{b_i-a_i}{2}\sum_{k=0}^Nu_k\Big(\sum_{j=0}^Nh_k(y(x_i,t_j))\gamma(x,y(x_i,t_j))\omega_j\Big)&-u_i\int_{-\delta}^\delta\gamma(|s|)ds\\
&=f(x_i)-\int_{\Lambda_c} g(y)\gamma(|y-x_i|)dy,\;0\leq i\leq N.
\end{split}
\end{equation}
\begin{rem}
It is worthwhile to point out that the collocation points $\{x_j\}_{j=0}^N$ and quadrature points $\{t_j\}_{j=0}^N$ could be chosen differently in type and number. As a result, we can also use Legendre-Gauss-Radau or Legendre-Gauss-Lobatto for the integral term.
\end{rem}
More precisely, we divide the integral range of $(x_i-\delta,x_i+\delta)$ into three case.
\begin{itemize}
\item {\rm{Case I:}} $ -1<x_i-\delta<x_i+\delta<1.$ 
For easy of implementation and analysis, we
convert the interval $[x_i-\delta,x_i+\delta]$ to $[-1,1]$ by a linear transformation:
\begin{equation}
y=x_i+\delta t,\quad t\in [-1,1].
\end{equation}
The scheme becomes
\begin{equation}\label{equnxi}
\delta\int_{-1}^{1}u_N(y(x_i,t))\gamma(x,y(x_i,t))dt-
u_N(x_i)\int_{-\delta}^\delta\gamma(|s|)ds
=f(x_i),\;0\leq i\leq N.
\end{equation}
Next, we approximate the integral term by a Legendre-Gauss-Lobatto type quadrature formula
with the notes and weights denoted by $\{t_j,\omega_j\}_{j=0}^M$, leading to
\begin{equation}\label{equnxisj}
\delta\sum_{j=0}^Mu_N(y(x_i,t_j))\gamma(x,y(x_i,t_j))\omega_j-u_N(x_i)\int_{-\delta}^\delta\gamma(|s|)ds
=f(x_i),\;0\leq i \leq N.
\end{equation}
Let $\{h_k\}^N_{k=0}$ be the Lagrange basis
polynomials associated with the Legendre-Gauss-lobatto-type points $\{x_i\}^N_{i=0}$. We expand the approximate solution $u_N$ as
\begin{equation}
u_N=\sum_{k=0}^Nu_kh_k(x).
\end{equation}
Inserting it into \eqref{equnxisj} leads to
\begin{equation}\label{eqxisj}
\delta\sum_{k=0}^Nu_k\sum_{j=0}^Mh_k(y(x_i,t_j))\gamma(x_i,y(x_i,t_j))\omega_j-u_i\int_{-\delta}^\delta\gamma(|s|)ds
=f(x_i),\;0\leq i \leq N.
\end{equation}
\item{\rm{Case II:}} $x_i-\delta<-1<x_i+\delta<1$, we have
\begin{equation}\label{case2}
\begin{split}
\int_{-1}^{x_i+\delta}u_N(y)\gamma(|y-x_i|)dy-
u_N(x_i)&\int_{-\delta}^\delta\gamma(|s|)ds\\
=&f(x_i)-\int_{x_i-\delta}^{-1}g(y)\gamma(|y-x_i|)dy,\;0\leq i\leq N.
\end{split}
\end{equation}
Next, we convert the interval $[-1,x_i+\delta,]$ to $[-1,1]$ and approximate the integral term by a Legendre-Gauss-Lobatto type quadrature $\{t_j,\omega_j\}_{j=0}^M$, leading to
\begin{equation}\label{CaseIIxi}
\begin{split}
\frac{x_i+\delta+1}{2}\sum_{j=0}^Mu_N(y(x_i,t_j))&\gamma(x_i,y(x_i,t_j))\omega_j-u_N(x_i)\int_{-\delta}^\delta\gamma(|s|)ds
\\=&f(x_i)+\int^{-1}_{x_i-\delta}
g(y)\gamma(|y-x_i|)dy,\;0\leq i \leq N.
\end{split}
\end{equation}
We expand the approximate solution $u_N$ as
\begin{equation}
u_N=\sum_{k=0}^Nu_kh_k(x).
\end{equation}
Plugging it into \eqref{CaseIIxi} leads to
\begin{equation}
\begin{split}
\frac{x_i+\delta+1}{2}\sum_{k=0}^Nu_k\sum_{j=0}^Mh_k(y(x_i,t_j))&\gamma(x_i,y(x_i,t_j))\omega_j-u_i\int_{-\delta}^\delta\gamma(|s|)ds
\\=&f(x_i)+\int_{x_i-\delta}^{-1}
g(y)\gamma(|y-x_i|)dy,\;0\leq i \leq N.
\end{split}
\end{equation}

\item {\rm{Case III:}} $-1<x_i-\delta<1<x_i+\delta,$ we have
\begin{equation}
\begin{split}
\int_{x_i-\delta}^{1}u_N(y)\gamma(|y-x_i|)dy-
u_N(x_i)&\int_{-\delta}^\delta\gamma(|s|)ds\\
=&f(x_i)-\int_{-1}^{x_i-\delta}g(y)\gamma(|y-x_i|)dy,\;0\leq i\leq N.
\end{split}
\end{equation}
Then, we can treat  this case in the same fashion as above.

\end{itemize}

}

\section{Maximum principle}\label{max}
The nonlocal diffusion operator is  analogous to its local counterpart $\Delta u$.
It is known that the  local diffusion equation  satisfies the maximum principle. The proposed nonlocal diffusion operator satisfies the following property.

\begin{thm}\label{maximum_model} {\rm (Maximum principle)}
	Suppose that ${\mathcal L_\delta}u$ is well-defined in $I:=[-1,1]$.
	If ${\mathcal L_\delta}u\geq 0$ in $I$, then a maximum of $u$ is attained in the interaction domain
	$I_c=(-\delta-1,-1)\cup (1,1+\delta)$.
\end{thm}

\begin{proof}
	Consider an auxiliary function $v(x)=u(x)+\epsilon e^{rx}$, where $\epsilon>0$, $r>0$, so
	\begin{eqnarray}\label{maximum_principle0}
	{\mathcal L_\delta}v(x)&=&{\mathcal L_\delta}u(x)+\epsilon e^{rx}\int_{-\delta}^{\delta} (e^{rs}-1)\gamma(|s|)\, ds\nonumber
	\end{eqnarray}
	We can easily get ${\mathcal L_\delta}v(x)>0$. We claim $v(x)$ cannot attains a maximumin $\Omega_s$. If not, assume that $v$ attains a nonnegative maximum at $x_0\in \Omega_s$, i.e., $v(x_0)=\max_{x\in\Omega} v(x)\geq 0$.  We have
	\begin{eqnarray}\label{maximum_principle1}
		{\mathcal L_\delta}v(x_0)&=&\int_{x_0-\delta}^{x_0+\delta} (v(y)-v(x_0))\gamma(|y-x_0|)\, dy\nonumber
	\end{eqnarray}
	Since  $v(y)-v(x_0)\leq 0$, it is easy to verify that the integral in \eqref{maximum_principle1} satisfies
	\begin{eqnarray}\label{mp_first_integral}
		&&\int_{x_0-\delta}^{x_0+\delta} \overbrace{(v(y)-v(x_0))}^{\leq 0}\overbrace{\gamma(|y-x_0|)}^{\geq 0}\, d y\leq 0. \nonumber
	\end{eqnarray}
Hence, we get ${\mathcal L_\delta}v(x_0)\leq 0$, which is a contradiction with the assumption of  ${\mathcal L_\delta}v(x_0)>0$. Let $\epsilon$ goes to 0, we can get the required result.
\end{proof}

\begin{lemma}\label{nolocalG}
If an integrable function $e(t)$ satisfies
\begin{equation}\label{Growneq}
e(t)\int_{-\delta}^\delta\gamma(|s|)ds-\int_{t-\delta}^{t+\delta}e(y)\gamma(|y-t|)dy=G(t),\quad t\in [-1,1],
\end{equation}
where $G(t)$ is an integrable function, then
	\begin{equation}
	\|e(t)\|_{L^\infty}(I)\leq C\|G\|_{L^\infty}(I).
	\end{equation}
\end{lemma}
\begin{proof} It is easy to see, from \eqref{Growneq}, that
\begin{equation*}
|e(t)|\int_{-\delta}^\delta\gamma(|s|)ds-\int_{t-\delta}^{t+\delta}|e(y)|\gamma(|y-t|)dy\leq \|G(t)\|_{L^{\infty}},
\end{equation*}
Denoting $E(t):=\frac{|e(t)|}{\|G(t)\|_{L^{\infty}(I)}}$, we rewrite above inequality as,
\begin{equation}\label{Et_inq}
\int_{t-\delta}^{t+\delta}E(y)\gamma(|y-x|)dy-E(t) \int_{-\delta}^\delta\gamma(|s|)ds \geq-1.\\
\end{equation}
For $\varpi(t):=t(1-t)/2+\delta(1+\delta), \; t\in [-1,1]$, it holds that
\begin{equation}\label{Omegafunc}
\int_{t-\delta}^{t+\delta}\varpi(y)\gamma(|y-t|)dy-\varpi(t) \int_{-\delta}^\delta\gamma(|s|)ds =- 1.\\
\end{equation}
In view of \eqref{Et_inq} and \eqref{Omegafunc}, we derive that
\begin{equation}
\int_{t-\delta}^{t+\delta}c(y)\gamma(|y-x|)dy-c(t) \int_{-\delta}^\delta\gamma(|s|)ds\geq 0, \; t\in[-1,1],
\end{equation}
where $c(t):=E(t)-\varpi(t).$ 
Since $e(t)=0,\;t\in(-1-\delta,-1)\cup(1,1+\delta),$  it holds that
\begin{equation}
c(t):=E(t)-\varpi(t)\leq 0,\; \;t\in(-1-\delta,-1)\cup(1,1+\delta).
\end{equation}
Consequently, according to Theorem \ref{maximum_model}, we have
\begin{equation*}c(t)\leq 0,\quad \forall t\in I,\quad i.e.,\;  E(t) \leq \|\varpi(t) \|_{L_{\infty}(I)}, \quad \forall t\in I.
\end{equation*}
Therefore,  let $C=\|\varpi(t) \|_{L_{\infty}(I)}$, we have
\begin{equation*}
	\|e(t)\|_{L^\infty}(I)\leq C\|G\|_{L^\infty}(I),\quad \forall t\in I.
\end{equation*}
\end{proof}

\section{Convergence Analysis}\label{con}

We now analyze the convergence of the scheme. For clarity of presentation,
we assume that the collocation and quadrature points in \eqref{eqxtj} are of the Legendre-Gauss-Lobatto type. The other cases can be treated in a similar fashion.

\subsection{Some Useful Lemmas}
\begin{lemma}\label{inter}{\cite[Lemma 4.8]{STW11}}
Assume that $u\in B_{-1,-1}^m(I)$ with $1\leq m\leq N+1,$ then for any $\phi\in {\mathbb P}_{N},$
Then the following estimates hold
\begin{equation}
\begin{split}
|(u,\phi)-\langle u,\phi\rangle_N|\leqslant c\sqrt{\frac{(N-m+1)!}{N!}}(N+m)^{-(m+1)/2}\|\partial_x^mu\|_{\omega^{m-1,m-1}}\|\phi\|,\\
\end{split}
\end{equation}
where $c$ is a positive constant independent of $m, N, \phi$ and $u.$
\end{lemma}
\begin{lemma}{\cite[B.33]{STW11}}
Let $(a,b)$ be a finite interval. There holds the Sobolev inequality:
\begin{equation}\label{ineq1}
\max_{x\in[a,b]}|u(x)|\leq \big( \frac{1}{b-a}+2\big)^{1/2}\|u\|_{L^2(a,b)}^{1/2}\|u\|^{1/2}_{H^1(a,b)},\quad \forall u\in H^1(a,b),
\end{equation}
which is also known as the {\emph{Gagliardo-Nirenberg interpolation inequality}}.
\end{lemma}

\begin{lemma}{\cite[B.44]{STW11}}
For any $u\in H^1(a,b)$ with $u(x_0) = 0$ for some $x_0\in (a,b)$, the following Poincare inequality holds:
\begin{equation}\label{ineq2}
\|u\|_{L^2(a,b)}\leq \|u'\|_{L^2(a,b)},\quad \forall u\in H^1(a,b).
\end{equation}
\end{lemma}

\begin{lemma}\label{LGLintp}{\cite[Theorem 3.44]{STW11}}
For any $u\in B^m_{-1,-1}(I)$, we have that for $1 \leq m \leq N + 1$,
\begin{equation}
\|\partial_x(I_Nu-u)\|+N\|I_Nu-u\|_{\omega^{-1,-1}}\leq c\sqrt{\frac{(N-m+1)!}{N!}}(N+m)^{(1-m)/2}\|\partial_x^m u\|_{\omega^{m-1,m-1}}.
\end{equation}
\end{lemma}

\subsection{Error analysis in $L_{\infty}$}
\begin{thm} Let $u$ be the exact solution of \eqref{eqxy}, and  assume that
\begin{equation}
\left \{ \begin{split}
&u_N=\sum_{k=0}^N u_k h_k(x), \quad \forall x\in I\\
&u_N=g(x), \qquad \forall x\in (-1-\delta,-1)\cup(1,1+\delta).
\end{split}\right.
\end{equation}
where $u_k$ is given by \eqref{coll} and $h_k(x)$ is the $k$-th Lagrange basis function associated with the
Gauss-points $\{x_k\}^N_{k=0}.$ If
\begin{equation}
\gamma(x,y)\in L^{\infty}(D)\cap L^{\infty}(I;B^k_{-1,-1}(I)),\;\partial_x\gamma(x,y)\in L^{\infty}(D),\;u\in B^m_{-1,-1}(I),
\end{equation}
where $D = \{(x,y): -1 \leq x,y \leq 1\}$ and $1 \leq k,m \leq  N + 1$. Then we have
\begin{equation}\label{uLinfinity}
\begin{split}
\|u-u_N\|_{L^\infty(I)}\leq&  c\sqrt{\frac{(N-k+1)!}{N!}}(N+k)^{-k/2}\|\partial_y^k\gamma(x,\cdot)\|_{\omega^{k-1,k-1}}\|u\|\\
&+c \sqrt{\frac{(N-m+1)!}{N!}}(N+m)^{-m/2}\|\partial_x^mu\|_{\omega^{m-1,m-1}}.
\end{split}
\end{equation}
provided that $N$ is sufficiently large, where $C$ is a constant independent of $N.$
\end{thm}
\begin{proof} Let $I_N$ be the Legendre-Gauss-Lobatto interpolation operator. We start from \eqref{eqxs} and reformulate  it as
\begin{equation}\label{eqI_points}
	\int_{x_i-\delta}^{x_i+\delta}u_N(y)\gamma(|y-x_i|)dy-
	u_i\int_{-\delta}^\delta\gamma(|s|)ds
	=f(x_i)+J_1(x_i),\;x\in I.
\end{equation}
where
\begin{equation}
J_1(x)=\delta\Big(\int_{-1}^1u_N(y(x_i,t))\gamma(x,y(x_i,t))dt-\sum_{j=0}^Nu_N(y(x_i,t_j))\gamma(x,y(x_i,t_j))\omega_j\Big).
\end{equation}
Multiply $h_i(x)$ on both sides, and take the summation from $0$ to $N$, we have,
\begin{equation}\label{eqIn}
I_N\left(\int_{x-\delta}^{x+\delta}u_N(y)\gamma(|y-x|)dy\right) -
u_N(x)\int_{-\delta}^\delta\gamma(|s|)ds
=I_N f+I_N J_1,\;x\in I.
\end{equation}
Clearly, by \eqref{eqs},
\begin{equation}\label{eqInf}
I_N f=I_N\left(\int_{x-\delta}^{x+\delta}u(y)\gamma(|y-x|)dy\right)-I_N u\int_{-\delta}^\delta\gamma(|s|)ds.
\end{equation}
Denote $e=u-u_N.$ Inserting \eqref{eqInf} into \eqref{eqIn} leads to the error equation:
\begin{equation}
e(x)\int_{-\delta}^\delta\gamma(|s|)ds-\int_{x-\delta}^{x+\delta}e(y)\gamma(|y-x|)dy=I_NJ_1+J_2(x)+J_3(x),
\end{equation}

where $e(y)=0$ for $y\in(-1-\delta,-1)\cup(1,1+\delta)$, and
\begin{equation}
\begin{split}
&J_2(x)=(u-I_Nu)\int_{-\delta}^\delta\gamma(|s|)ds,\\
&J_3(x)=\int_{x-\delta}^{x+\delta}e(y)\gamma(|y-x|)dy-I_N\left(\int_{x-\delta}^{x+\delta}e(y)\gamma(|y-x|)dy\right).
\end{split}
\end{equation}
According to Lemma \ref{nolocalG},
\begin{equation}\label{Gterm}
\|e(x)\|_{L^{\infty}(I)} \leq C(\|(I_NJ_1)\|_{L^{\infty}(I)}+ \|J_2\|_{L^{\infty}(I)}+ \|J_3\|_{L^{\infty}(I)}).
\end{equation}
It remains to estimate the three terms on the right hand side of \eqref{Gterm}. 
By Lemma \ref{inter},
\begin{equation}
\begin{split}
|J_1(x)|&=|\delta\Big(\int_{-1}^1u_N(y(x,\cdot))\gamma(x,y(x,\cdot))dt-\sum_{j=0}^Nu_N(y(x,t_j))\gamma(x,y(x,t_j))\omega_j\Big)|\\
&\leq  c\sqrt{\frac{(N -k+ 1)!}{N!}}(N+k)^{-\frac{k+1}{2}}\times\delta\|\partial_t^k\gamma(x,y(x,\cdot))\|_{\omega_{k-1,k-1}}\|u_N(y(x,\cdot))\|
\end{split}
\end{equation}
A direct calculation yields
\begin{equation}
\begin{split}
\|\partial_t^k\gamma(\cdot)\|^2_{\omega^{k-1,k-1}}&=\int_{-1}^1|\partial_t^k\gamma(x,y(x,t))|^2(1-t^2)^{k-1}dt\\
&=\delta\int_{x-\delta}^{x+\delta} |\partial_y^k\gamma(x,y))|^2(x+\delta-y)^{k+1}(\delta-x+y)^{k+1}dy\\
&\leq\|\partial_y^k\gamma(x,\cdot)\|^2_{\omega^{k-1,k-1}},
\end{split}
\end{equation}
and
\begin{equation}
\delta\|u_N(y(x,\cdot)\|^2=\int_{x-\delta}^{x+\delta}|u_N(y)|^2dy\leq \|u_N\|^2.\\
\end{equation}
Hence, we obtain the estimate of $|J_1|:$
\begin{equation}
|J_1(x)|\leq c\sqrt{\frac{(N-k+1)!}{N!}}(N+k)^{-(k+1)/2}\|\partial_y^k\gamma(x,\cdot)\|_{\omega^{k-1,k-1}}\|u_N\|.
\end{equation}
In what follows, we  use the asymptotic estimate of the Lebesgue constant (cf,\cite{QW88}):
\begin{equation}
\max_{|x|\leq 1}\sum_{j=0}^N|h_j(x)|\simeq \sqrt{N}, \; N\gg 1.
\end{equation}
This implies
\begin{equation}\label{INJ1}
\begin{split}
\|I_NJ_1\|_{L^\infty}&\leq c\|J_1\|_{L^\infty}\max_{|x|\leq 1}\sum_{j=0}^N|h_j(x)|\\
&\leq c\sqrt{\frac{(N-k+1)!}{N!}}(N+k)^{-k/2}\|\partial_y^k\gamma(x,\cdot)\|_{\omega^{k-1,k-1}}\|u_N\|\\
&\leq c\sqrt{\frac{(N-k+1)!}{N!}}(N+k)^{-k/2}\|\partial_y^k\gamma(x,\cdot)\|_{\omega^{k-1,k-1}}(\|e\|+\|u\|).
\end{split}
\end{equation}
Using the inequalities \eqref{ineq1} and \eqref{ineq2}, we obtain from Theorem \ref{LGLintp} that
\begin{equation}\label{J2}
\begin{split}
\|J_2\|_{L^{\infty}}&\leq c\|u-I_Nu\|^{1/2}\|\partial _x(u-I_Nu)\|^{1/2}\\
&\leq c \sqrt{\frac{(N-m+1)!}{N!}}(N+m)^{-m/2}\|\partial_x^mu\|_{\omega^{m-1,m-1}}.
\end{split}
\end{equation}
Moreover, using Theorem \ref{LGLintp} with $m = 1$ yields
\begin{equation}
\begin{split}
\|J_3\|\leq & cN^{-1}\|\gamma(x,x+\delta)e(x+\delta)-\gamma(x,x-\delta)e(x-\delta)+\int_{x-\delta}^{x+\delta}\partial_x\gamma(x,y)e(y)dy\|\\
\leq & cN^{-1}\Big(\max_{|x\pm \delta|\leq 1}|\gamma(x,x+\delta)|+\max_{D}\|\partial_x\gamma(x,y)\|_{L^\infty}\Big)\|e\|.\\
\end{split}
\end{equation}
Then, we have
\begin{equation}\label{J3}
\begin{split}
\|J_3\|_{L^{\infty}}&\leq c \|J_3\|^{1/2}\|\partial_xJ_3\|^{1/2}\leq cN^{-1/2}\|e\|^{1/2}\|\partial_x\int_{x-\delta}^{x+\delta}\gamma(x,y)e(y)dy\|\\
&\leq cN^{-1/2}\|e\|\leq cN^{-1/2}\|e\|_{L^\infty}.
\end{split}
\end{equation}
Finally, a combination of \eqref{INJ1}, \eqref{J2} and \eqref{J3} leads to the estimate .
\end{proof}

\section{Numerical Experiments}\label{exp}
Without lose of generality, we will only use the Legendre-Gauss-Lobatto points  as the collocation points. Our numerical evidences show that the other
two kinds of Legendre-Gauss points produce results with similar accuracy.

{\bf{Example 1.}} We first consider the  equation (2.1) with
\begin{equation}
\gamma(x)=\frac{3}{\delta^3},\;f_\delta(x)=-\frac{6e^{4x}}{\delta^3}-\frac{3e^{4x}(e^{-4\delta}- e^{4\delta})}{4\delta^3}
\end{equation}
The corresponding exact solution is given by $u(x) = e^{4x}$.\\

\begin{figure}[!ht]\label{exp1}
  \begin{center}
    \includegraphics[width=0.43\textwidth]{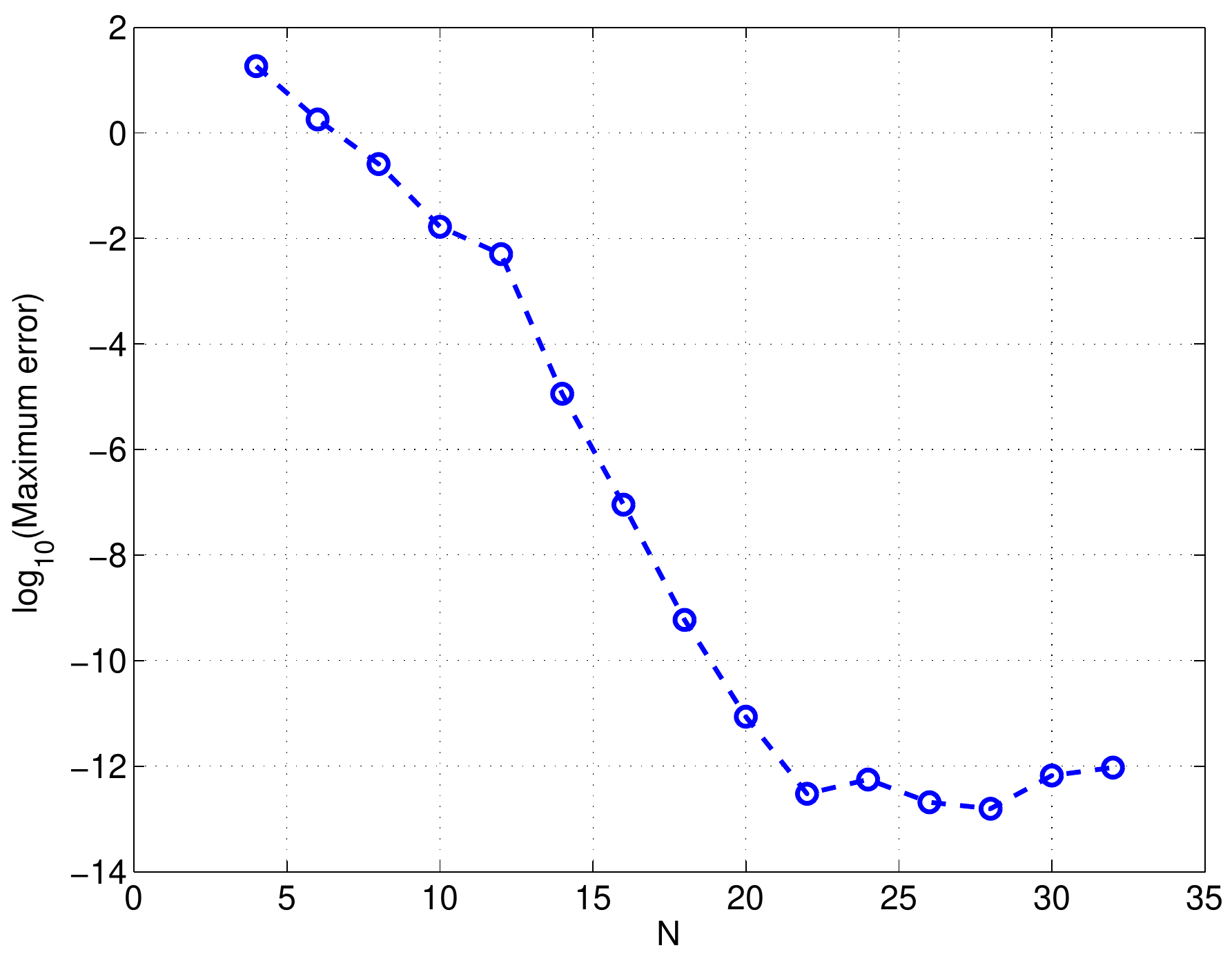}
    \includegraphics[width=.45\textwidth]{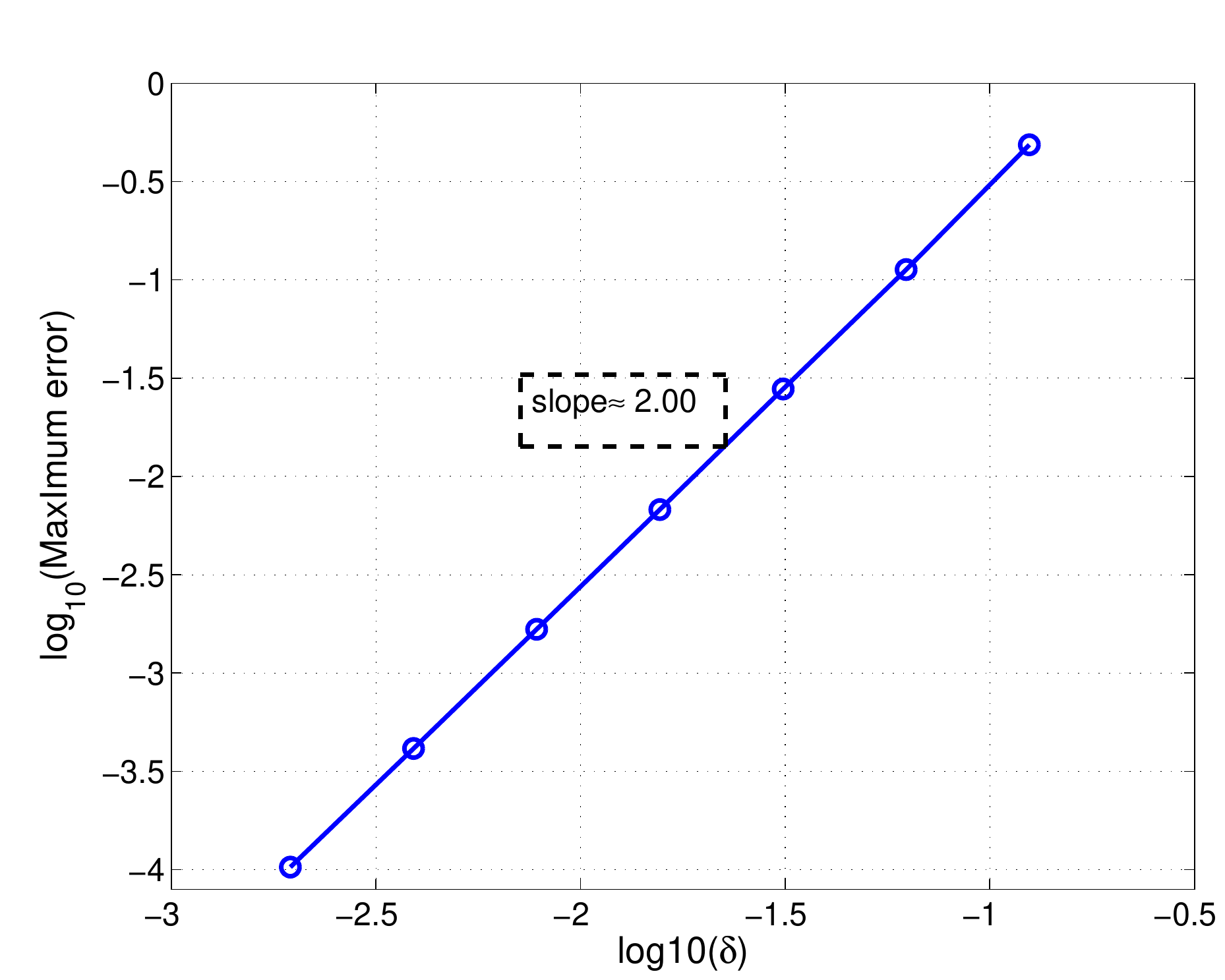}
  \end{center}
\caption{\small Left: Maximum error for the example 1; Right: $\log_{10}$(maximum error) against $\log_{10}(\delta)$ for N=64.}\label{exp1}
  \end{figure}
In Figure \ref {exp1} (left), we plot $\log_{10}$ (Maximum error)  against $N \in [4,32],$ and observe that the desired spectral accuracy is obtained.
We plot in Figure \ref {exp1} (right), $\log_{10}$ (Maximum error)  against $\log_{10}(\delta).$ The slopes are nearly equal to $2$. These results indicate that when we use set a relatively high number of collocation points,($N=64$), the error from spectral collocation discretization is relatively negligible. The error left is only the error between nonlocal diffusion models and local diffusion models, which is order of $2$ for $\delta$.

\section{Conclusions}

In this paper, we presented a spectral method for a nonlocal diffusion model and provide a rigorous error analysis which
theoretically justifies the spectral rate of convergence provided that the kernel function and the source function are sufficiently smooth.

We mainly focus on one-dimension case in this paper, and there is no difficulty to extent this algorithm to a higher dimension, especially it will be obvious in spectral collocation methods, the numerical integration to assemble stiff matrices is more accurate. As is seen in the paper, to apply the spectral collocation methods, the kernel function cannot be too singular.  In the future we will focus on how to develop spectral collocation methods to deal with nonlocal diffusion models with singular kernels.

\bibliographystyle{plain}
\bibliography{Thesisreferpapers}

\end{document}